\documentclass[10pt]{amsart}
\usepackage{amsfonts}
\usepackage{indentfirst,amsthm,bm,amsmath}
\usepackage{hyperref}
\newtheorem{theorem}{Theorem}[section]

\newtheorem{lemma}[theorem]{Lemma}

\newtheorem{remark}[theorem]{Remark}

\begin{document}
\title[]{On the extension of analytic solutions of a class of first-order q-difference equations}

\author[W. L. Liu]{Wenlong Liu}
\address[Wenlong Liu]{Department of Physics and Mathematics, University of Eastern Finland, P.O.Box 111, FI-80101 Joensuu, Finland}
\email{wenlong.liu@uef.fi}

\thanks{The author is supported by the China Scholarship Council (\# 202306820018).}


%\date{}

%\dedicatory{}


\begin{abstract}
In this paper, we use the Banach fixed point theorem to examine the existence of  meromorphic solutions to the following first-order $q$-difference equation
\begin{align}\tag{†}\label{dagger}
y(qz)=\frac{a_1(z)y(z)+a_2(z)y(z)^2+\dots+a_p(z)y(z)^p}{1+b_1(z)y(z)+\cdots +b_t(z)y(z)^t}, 
\end{align}
where $q\in \mathbb{C},$ $a_1(z), \dots, a_p(z); b_1(z), \dots, b_t(z)$ are all meromorphic functions. We establish sufficient conditions ensuring the existence and uniqueness of meromorphic solutions that can be extended to the entire complex plane $\mathbb{C}.$

More precisely, we have the following result. If $\left | q \right |\geq 3 $ and \[|a_1(z)| = \max_{1 \le j \le p} |a_j(z)| \le \frac{1}{|z|}, 
\quad 
\max_{1 \le k \le t} |b_k(z)| \le \frac{1}{|z|}, 
\quad 
z \in \{\, |\Re(z)| \ge \rho > 0 \,\},
\] and $y(0)\ne \infty,$
then we prove that~\eqref{dagger} admits a unique meromorphic solution in $D(\rho),$ which can be extended meromorphically to $\mathbb {C}.$
Moreover, if $a_1(z)\equiv 0,$  the conclusion still holds. Furthermore,
if $\left | q \right |\geq 6$ and
\begin{gather*}
|a_1(z)| \le \frac{1}{|q|}, \quad 
|a_j(z)| \le |q|^{|z|} \quad (2 \le j \le p), \quad 
|b_k(z)| \le |q|^{|z|} \quad (1 \le k \le t), \\[4pt]
z \in D(\rho,\sigma)
 = \{\, z : |\Re(z)| \le \rho,\; |\Im(z)| \le \sigma, \,\, \rho>0,\,\, \sigma>0 \,\},
\end{gather*}
and $y(0)\ne \infty,$
then we prove that \eqref{dagger} admits a unique meromorphic solution in $D(\rho, \sigma),$ which can also be extended meromorphically to $\mathbb {C}.$ This conclusion remains valid in the case where $a_1(z)\equiv 0.$
\end{abstract}

\maketitle

\section{Introduction}
The classical paper on the  Schr\"{o}der equation
\begin{align}\label{E44} 
y(qz) = R(y(z)),
\end{align}
where $q\in \mathbb{C}\setminus \left \{ 0,1 \right \},$ and $R(y)$ is a rational function in $y(z)$, is due to J. Ritt~\cite{Ritt1926}. L. Rubel~\cite{Rubel1983} posed the question: What can be said about the more general equation 
$$y(qz)=R(z,y(z)),$$
where $R(z,y(z))$ is rational in both $y(z)$ and $z?$ H.~Wittich~\cite{Wittich1949} studied the equation
\begin{align}\label{E45}
y(qz)=a(z)y(z)+b(z),
\end{align}
where the coefficients are polynomials, while K.~Ishizaki~\cite{Ishizaki1998} considered the same equation and assumed that the coefficients are meromorphic functions. G. Gundersen, J. Heittokangas, I. Laine et al.~\cite{Gundersen2002} studied the equation
$$y(qz)=R(z,y(z))=\frac{\sum_{j=0}^{p}a_j(z)y(z)^{j} }{\sum_{j=0}^{q}b_j(z)y(z)^{j} },$$
where the coefficients $a_j (z), b_j(z)$ are meromorphic functions and $q$ is a complex constant, and they showed that a number of results on the growth and value distribution of solutions by using Nevanlinna theory. M. J. Ablowitz, R. Halburd, and B. Herbst \cite{Ablowitz2000} used Nevanlinna theory to study
the integrability of difference equations and they showed that the existence of sufficiently
many finite-order meromorphic solutions of a difference equation is a good candidate for
a difference analogue of the Painlev\'e property (An ordinary differential equation is said to possess
the Painlev\'e property \cite{Ablowitz1991} if all of its solutions are single-valued about all movable
singularities.) R. Korhonen and Y. Zhang  \cite{Korhonen2024} used Nevanlinna theory to studied the equation \begin{align}\label{E41} 
y(qz)^{n}=R(z,y(z))
\end{align} and they showed if \eqref{E41} has a zero order transcendental meromorphic solution, then \eqref{E41} reduces
to a $q$-difference linear or Riccati equation, or to an equation that can be transformed to
a $q-$difference Riccati equation.
Nevanlinna theory \cite{Hayman1964, Laine1993} is an important tool, which is widely used to study complex $c-$difference and $q-$difference equation. It is necessary to restrict the growth of a meromorphic solution when applying Nevanlinna theory. In contrast, our approach is based on the Banach fixed point theorem, which allows us to establish existence results without such growth conditions.

R. Halburd, R.  Korhonen ~\cite{Halburd2006} studied the existence of meromorphic solutions of the constant coefficient
first-order nonlinear difference equation
\begin{align}\label{E46} 
y(z+1)=R(y(z))
\end{align}
based on the Banach fixed point theorem. S. Shimomura \cite{shimomura1981entire} had earily treated \eqref{E46} when $R(y(z))$ is reduced to a polynomial. R. Halburd, R. Korhonen, Y. Liu, T. Meng \cite{Halburd2025} used the Banach fixed point theorem to further
studied the existence of solutions in the complex domain of
first-order difference equations of the form \begin{align}\label{E47} 
y(z+1)=R(z,y(z)).
\end{align}
In particular, they studied a meromorphic
solutions to the non-autonomous difference equation
\begin{align}\label{E42}
y(z+1)=\frac{\lambda y(z)+a_2(z)y(z)^2+\dots+a_p(z)y(z)^p}{1+b_1(z)y(z)+\cdots +b_t(z)y(z)^t}, 
\end{align}
where $\lambda \in \mathbb{C}\setminus \left \{ 0 \right \}  ,$ and $a_{j}(z), b_k(z)$ are meromorphic functions. N. Yanagihara \cite{yanagihara1980meromorphic} had earlier investigated the value distribution behaviors of the meromorphic solutions of equation \eqref{E46} and \eqref{E47}.

In this paper we consider the equation 
\begin{align}\label{E43}
y(qz)=\frac{a_1(z)y(z)+a_2(z)y(z)^2+\dots+a_p(z)y(z)^p}{1+b_1(z)y(z)+\cdots +b_t(z)y(z)^t}=R(z,y(z)), 
\end{align}
where $q\in \mathbb{C},$ $a_1(z), \dots, a_p(z); b_1(z), \dots, b_t(z)$ are all meromorphic functions. Equation \eqref{E43} can be referred as a corresponding $q-$difference version of \eqref{E42}. 
We can easily find that the right-hand side of \eqref{E43} satisfies $R(z,0)=0.$ This implies that $y(z)\equiv 0$ is a fixed point of $R(z,y(z))$ such that we can transform it to more convenient form. The results obtained in this paper complement and extend previous works on complex difference equations by providing an analytic existence framework for $q$-difference equations with meromorphic coefficients.

The remainder of this paper is organized as follows. Section $2$ presents the main results and its remark. Sections $3-6$ contain the proofs of the four main results. Section $7$ gives acknowledgements.

\section{main results}

\begin{theorem}\label{T1}
Let $q\in \mathbb{C}$ satisfy $\left | q \right |\geq 3 ,$ and let $a_1(z),$ \dots, $a_p(z);$ $b_1(z),$ \dots, $b_t(z)$ be meromorphic functions such that there exists a $\rho >0$ for which \begin{align*}
&\left | a_1(z) \right |  = \max \left \{ {\left | a_j(z) \right | }  \right \}  = \lambda\leq \frac{1}{\left | z \right | } , \,\,\,\left |  {b_k(z)}  \right |\leq \frac{1}{\left | z \right | }, \,\,\,1\leq j\leq p,\,\,\,1\leq k\leq t,    \\   
&z\in D(\rho)  = \left \{ z:\left | \Re(z) \right | \geq \rho  \right \}. 
\end{align*} 
Assume that $y(0)\ne \infty.$
Then 
\begin{align}\label{E1} 
y(qz)=\frac{a_1(z)y(z)+a_2(z)y(z)^2+\dots+a_p(z)y(z)^p}{1+b_1(z)y(z)+\cdots +b_t(z)y(z)^t} 
\end{align}
has a unique meromorphic solution in $D(\rho)$ satisfying  \begin{align}\label{E18} 
 y\left ( \frac{z}{q^m}  \right )=\frac{1}{\left | q \right |^m }y(z),\,\,\,m\in \mathbb{N},\,\,\, z\in D(\rho).    
\end{align}
\end{theorem}
\begin{remark}

(1). The condition $y(0)\ne \infty$ is essential. Together with \eqref{E18} and $y(0)\ne \infty$ guarantee that  the local meromorphic solution to \eqref{E1} in $D(\rho)$ can be extended meromorphically to the entire complex plane $\mathbb{C}.$
In general, a local meromorphic solution to \eqref{E1} in $D(\rho)$ cannot be extended meromorphically to the whole complex plane $\mathbb{C}$ directly. H.
Poincar´e \cite[p. 317]{H1890} provided the following example. If $\left | q \right |>1 ,$ then the function $$f(z)=\sum_{n=-\infty}^{\infty}\left ( \frac{1}{q}  \right )^{n{^2}}z^{2n}$$
is analytic for all $z\ne 0$ and has an essential singularity 
at $z=0.$ Moreover, it satisfies the equation
$$f(qz)=qz^2f(z).$$

(2). On the other hand, we can find a conformal mapping that allows us to extend a local meromorphic solution meromorphically  to the whole complex plane $\mathbb{C}$ without assuming \eqref{E18}. By Combining the Cayley transform, we find that $w=\phi (z)=\frac{z+\rho+1}{z+\rho -1}$ maps $$\overline D(\rho )=:\left \{ z:\Re(z)<-\rho  \right \}\to \mathbb{D}=:\left | w \right |<1  $$ 
and its inverse mapping
$z=\psi (w)=\frac{\rho (1-w)+1+w}{w-1}$ maps
$$ \mathbb{D}=:\left | w \right |<1\to \overline D(\rho )=:\left \{ z:\Re(z)<-\rho  \right \}.$$
Hence, we can first map $\overline D(\rho)$ onto $\mathbb{D}$ by $w=\phi (z),$
and then, by Lemma~\ref{L4},  extend the local meromorphic solution to \eqref{E1}
meromorphically to the whole complex plane $\mathbb{C}.$ In this case, the local meromorphic solution to \eqref{E1} can be expressed as $Y(w)=y\left ( \psi (w) \right ),$ and all the coefficient functions $A_j(w)=a_j(\psi(w) ), B_k(w)=b_k(\psi(w))$ are meromorphic functions in $\mathbb{D}.$
\end{remark}

In particular, if $a_1(z)\equiv 0,$ then we have the following result.
\begin{theorem}\label{T2}
Let $q\in \mathbb{C}$ satisfy $\left | q \right |\geq 3 ,$ and let $a_2(z),$ \dots, $a_p(z);$ $b_1(z),$ \dots, $b_t(z)$ be meromorphic functions such that there is a $\rho >0$ for which \begin{align*}
&   {\left | a_j(z) \right | }   \leq \frac{1}{\left | z \right | } ,\,\,\, \left |  {b_k(z)}  \right |\leq \frac{1}{\left | z \right | }, \,\,\,2\leq j\leq p,\,\,\,1\leq k\leq t,    \\   
&z\in D(\rho)  = \left \{ z:\left | \Re(z) \right | \geq \rho  \right \}. 
\end{align*} 
Assume that $y(0)\ne \infty.$
Then 
\begin{align}\label{E11} 
y(qz)=\frac{a_2(z)y(z)^2+\dots+a_p(z)y(z)^p}{1+b_1(z)y(z)+\cdots +b_t(z)y(z)^t} 
\end{align}
has a unique meromorphic solution in $D(\rho)$ satisfying   \begin{align}\label{E13} 
y\left ( \frac{z}{q}  \right )=\frac{1}{\left | q \right | }y(z),\,\,\,z\in D(\rho).
\end{align}
\end{theorem}

Next, we study the existence of  local meromorphic solutions to \begin{align*}
y(qz)=\frac{a_1(z)y(z)+a_2(z)y(z)^2+\dots+a_p(z)y(z)^p}{1+b_1(z)y(z)+\cdots +b_t(z)y(z)^t} 
\end{align*}
in $$D(\rho,\sigma)=\left \{ z:\left | \Re(z) \right |\leq \rho,\left | \Im(z) \right |\leq \sigma   \right \},$$  
which differs from the domain $D(\rho)=\left \{ z:\left | \Re(z) \right |\geq \rho  \right \}$ that appears in Theorems~\ref{T1} and~\ref{T2}.
In Theorem~ \ref{T1} and ~\ref{T2}, we can extend a local meromorphic solution to the whole complex plane under the assumptions \eqref{E18} and \eqref{E13}, respectively. In the following theorems, since the local meromorphic solution is considered in the domain $D(\rho, \sigma),$ we cannot extend it meromorphically to the entire complex plane in the same way. However, we can extend the local meromorphic in $D(\rho, \sigma)$ meromorphically
to the entire complex plane $\mathbb{C}$  by using Lemma~\ref{L4}.

Moreover, we assume that all the coefficient functions $a_j(z)$ and $b_k(z)$ are bounded in a manner different from that in
Theorems~\ref{T1} and~\ref{T2}.

\begin{theorem}\label{T3}
Let $q\in \mathbb{C}$ satisfy $\left | q \right |\geq 6,$ and let $a_1(z),\dots,a_p(z);$ $b_1(z),\dots,b_q(z)$ be meromorphic functions such that, there are $\rho >0, \sigma >0$ for which
\begin{align}\label{E30}
\lvert a_1(z) \rvert = \lambda \le \frac{1}{\lvert q \rvert}, \,\,
\lvert a_j(z) \rvert \le \lvert q \rvert^{\lvert z \rvert}, \,\,
\lvert b_k(z) \rvert \le \lvert q \rvert^{\lvert z \rvert},\,\, 
2 \le j \le p, \ 1 \le k \le t,
\end{align} \begin{align*}
z\in D(\rho,\sigma)=\left \{ z:\left | \Re(z) \right |\leq \rho,\left | \Im(z) \right | \leq \sigma   \right \}. 
\end{align*}
Assume that $y(0)\ne \infty.$
Then \begin{align}\label{E40}
y(qz)=\frac{a_1(z)y(z)+a_2(z)y(z)^2+\dots+a_p(z)y(z)^p}{1+b_1(z)y(z)+\cdots +b_t(z)y(z)^t} 
\end{align} has a unique meromorphic solution in $\mathbb{C}$ satisfying \begin{align}\label{E23}
y\left ( \frac{z}{q^m}  \right )=\frac{1}{\left | q \right |^m }y(z),\,\,\,m\in\mathbb{N},\,\,\,z\in D(\rho).   
\end{align}
\end{theorem}

Similarly, in the case where $a_1(z)\equiv 0,$  we present the following result. 
\begin{theorem}\label{T4}
Let $q\in \mathbb{C}$ satisfy $\left | q \right |\geq 6,$ and let $a_2(z),\dots,a_p(z);$ $b_1(z),\dots,b_q(z)$ be meromorphic functions such that, there are $\rho >0, \sigma >0$ for which
\begin{align}\label{E32}
\lvert a_j(z) \rvert \le \lvert q \rvert^{\lvert z \rvert},\,\, 
\lvert b_k(z) \rvert \le \lvert q \rvert^{\lvert z \rvert},\,\, 
2 \le j \le p, \ 1 \le k \le t,
\end{align} \begin{align*}
z\in D(\rho,\sigma)=\left \{ z:\left | \Re(z) \right |\leq \rho,\left | \Im(z) \right | \leq \sigma   \right \}. 
\end{align*}
Assume that $y(0)\ne \infty.$
Then \begin{align}\label{E39}
y(qz)=\frac{a_2(z)y(z)^2+\dots+a_p(z)y(z)^p}{1+b_1(z)y(z)+\cdots +b_t(z)y(z)^t} 
\end{align} has a unique meromorphic solution in $\mathbb{C}$ satisfying \begin{align}\label{E33}
y\left ( \frac{z}{q}  \right )=\frac{1}{\left | q \right | }y(z),\,\,\,z\in D(\rho). 
\end{align}    
\end{theorem}

\section{proof of theorem \ref{T1}}

\begin{lemma}[\textit{Banach fixed point theorem}]\label{L2}
Let $T$ be a contraction mapping defined on a complete metric space $X.$ Then $T$ has a unique fixed point in $X.$
\end{lemma}

\begin{proof}
$\mathbf{Step 1:}$ We will now transform equation \eqref{E1} into a more convenient form. Let $\rho=\left | q \right |\geq 3.$ Then all the coefficient functions $a_j(z),$ $b_k(z)$ have no poles in $D(\rho).$
Let $I$ denote a set of multi-indices $\tau =(j_1,\dots,j_t),$ where $j_k \in  \mathbb{N} \cup \left \{ 0 \right \} $ for all $k=1, \dots , t.$ The degree of $\tau$ is defined to be
$$d(\tau)=j_1+j_2+\dots+j_t$$
and the weight of $\tau$ is $$w(\tau)=j_1+2j_2+\dots+tj_t.$$
Let $R(z,y(z))=\frac{a_1(z)y(z)+a_2(z)y(z)^2+\dots+a_p(z)y(z)^p}{1+b_1(z)y(z)+\cdots +b_t(z)y(z)^t}.$ Then $R(z,0)=0.$ This implies $y(z)=0$ is a fixed point of $R(z,y(z)).$ By combining the Multinomial theorem, we obtain
\begin{align}\label{E2}
y(qz)&=R(z,y(z))\\\nonumber 
&=\left ( \sum_{j=1}^{p}a_j(z)y(z)^j  \right )\frac{1}{1+\sum_{k=1}^{t}b_k(z)y(z)^{k}}\\\nonumber
&=\left ( \sum_{j=1}^{p}a_j(z)y(z)^j  \right )\sum_{n=0}^{\infty}\left ( -\sum_{k=1}^{t}b_k(z)y(z)^k  \right )^n\\\nonumber  
&= \sum_{j=1}^{p}\sum_{n=0}^{\infty}\sum_{\tau \in I }(-1)^n\frac{n!}{j_1!j_2!\dots j_t!}a_j(z)b_1(z)^{j_1} b_2(z)^{j_2}\dots b_t(z)^{j_t}y(z)^{w(\tau)+j}, \\\nonumber  
\end{align} 
which is Taylor series of $R(z,y(z))$ around the fixed point $\gamma=0.$ 
Note that $ \left | a_1(z) \right | \equiv \lambda,$ then $a_1(z)=\lambda e^{i\theta }.$ 
In the following, we assume $a_1(z)\equiv\lambda.$ Other cases  can be discussed similarly.
By \eqref{E2}, we can find that 
\begin{align}\label{E3}
y(qz)-\lambda y(z)=\sum_{j=2}^{\infty} c_j(z)y(z)^j,  
\end{align}
where the coefficients $c_j(z)$ are meromorphic functions of $z,$ depending on $a_j(z),$ $b_k(z).$ Also, together with \eqref{E2} and \eqref{E3}, we have the following:
\begin{align*}
\begin{cases}
 c_2(z)=a_2(z)-a_1(z)b_1(z)\\
 c_3(z)=a_3(z)-a_2(z)b_1(z)-a_1(z)b_2(z)+a_1(z)b_1(z)^2\\
 c_4(z)=a_4(z)-a_3(z)b_1(z)+a_2(z)b_1(z)^2
 -a_2(z)b_2(z)-a_1(z)b_1(z)^3\\
 \phantom{c_4(z)=}\,+a_1(z)b_1(z)b_2(z)+a_1(z)b_1(z)b_2(z)-a_1(z)b_3(z)\\
\dots 
\end{cases}
\end{align*}
Inductively, we have that $c_j(z)$ consists of $2^{j-1}$ terms, and the only term in $c_j(z)$ of degree $j$ is $\pm a_1(z)b_1(z)^{j-1}
.$ Then we have

\begin{align}\label{E7}
\left | c_j(z) \right |&\leq \sum_{1\leq k\leq j-1}\underset{i  = 1,\dots,p}{\max}\left | a_i(z) \right |\underset{i  = 1,\dots,t}\max\left | b_i(z) \right |^{k-1}+\left | a_1(z) \right |\left | b_1(z) \right |^{j-1}\\\nonumber
&\leq  \left ( 2^{j-1}-1 \right )\frac{1}{\left | z \right | }+\frac{1}{\left | z \right |^2 }\\\nonumber
&<2^{j-1}\frac{1}{\left | z \right | }.          
\end{align}
By equation \eqref{E3}, we have 
\begin{align}\label{E4}
\left\{
\begin{array}{l}
y(qz) - \lambda y(z) = \displaystyle\sum_{j=2}^{\infty} c_j(z)\, y(z)^j, \\[6pt]
\lambda y(z) - \lambda^2 y\!\left( \tfrac{z}{q} \right)
   = \displaystyle\sum_{j=2}^{\infty} \lambda\, c_j\!\left( \tfrac{z}{q} \right)
     y\!\left( \tfrac{z}{q} \right)^j, \\[6pt]
\lambda^2 y\!\left( \tfrac{z}{q} \right) - \lambda^3 y\!\left( \tfrac{z}{q^2} \right)
   = \displaystyle\sum_{j=2}^{\infty} \lambda^2 c_j\!\left( \tfrac{z}{q^2} \right)
     y\!\left( \tfrac{z}{q^2} \right)^j, \\[4pt]
\vdots
\end{array}
\right.
\end{align}
By adding together all the equation in \eqref{E4}, we obtain \begin{align}\label{E5}
y(qz)-\lambda^{k+1}y\left ( \frac{z}{q^k}  \right )=\sum_{j=2}^{\infty}\sum_{m=0}^{k}\lambda^mc_j\left ( \frac{z}{q^m}  \right )y\left ( \frac{z}{q^m}  \right )^j.      
\end{align}
Note that $\rho=\left | q \right |,$ then 
\begin{align}\label{E12}
\left | z \right |\geq \left | \Re(z) \right |\geq \rho=\left | q \right | \geq 3.
\end{align}  
Thus, $\lambda\leq\frac{1}{3}.$ Since $y(0)\ne \infty,$
letting $k\to\infty $ in \eqref{E5} yields that \begin{align}\label{E6} 
y(qz)=\sum_{j=2}^{\infty} \sum_{m=0}^{\infty}\lambda^mc_j
\left ( \frac{z}{q^m}  \right )y\left ( \frac{z}{q^m}  \right )^j. 
 \end{align}
Then $$y(z)=\sum_{j=2}^{\infty}\sum_{m=0}^{\infty} \lambda^mc_j\left ( \frac{z}{q^{m+1}}  \right )y\left ( \frac{z}{q^{m+1}}  \right )^j.$$

$\mathbf{Step 2:}$
We next define a Banach space $X$ and  prove that $T$ maps $X$ into itself.
Define an operator $T$ by 
\begin{align}\label{E8}
T[y](z)=\sum_{j=2}^{\infty}\sum_{m=0}^{\infty} \lambda^mc_j\left ( \frac{z}{q^{m+1}}  \right )y\left ( \frac{z}{q^{m+1}}  \right )^j=\sum_{j=2}^{\infty}\sum_{m=0}^{\infty}P_{mj}(z) .
\end{align}
Let $X$ be the set of all functions $z\to y(z),$ analytic and bounded in $$D(\rho)=\left \{ z:\left | \Re(z) \right | \geq \rho  \right \},$$ and let 
\begin{align}\label{E22} 
\left | \left | y \right |  \right |=\underset{z\in D(\rho)}{\sup }\left | y(z) \right |\leq \frac{1}{\left | q \right | }.
\end{align}
Clearly, $X$ is a metric space under the sup-norm and is a vector space. Moreover, if a sequence $\left \{ y_n(z) \right \} $ of $X$ is Cauchy, then $\lim_{n \to \infty}y_n(z)=y(z)$ exists uniformly in $D(\rho).$ Furthermore, Weierstrass's theorem shows that $y(z)$ is analytic in $D(\rho),$ and $\forall \epsilon >0,$ there exists $N\in \mathbb{N}$ such that $$\left | \left | y \right |  \right |=\left | \left | y-y_N+y_N \right |  \right |\leq \left | \left | y-y_N \right |  \right |+\left | \left | y_N \right |  \right |<\epsilon+\frac{1}{\left | q \right | }  .$$
Since $\epsilon$ is an arbitrarily positive number, we have $\left | \left | y \right |  \right |\leq \frac{1}{\left | q \right | } .$ 
This means that $X$ is a complete metric space. Note that $X$ is also a vector space, then $X$ is a Banach space. Assume that \eqref{E18} holds.
By combining \eqref{E18}, \eqref{E7}, \eqref{E12}, \eqref{E8}, \eqref{E22}, \begin{align}\label{E14}
\frac{1}{\left | q \right |^j-1 }\leq \frac{2}{\left | q \right |^j } ,\,\,\,\left | q \right |\geq 2, j\geq 2,
\end{align}
and
\begin{align}\label{E31}
\frac{4}{\left | q \right |^4-2\left | q \right |^2  }\leq \frac{1}{\left | q \right | },\,\,\, \left | q \right |\geq 2, 
\end{align}
we have 
\begin{align}\label{E9}
\left | T[y](z) \right |  =& \left | \sum_{j  = 2}^{\infty}\sum_{m  = 0}^{\infty}P_{mj}(z)   \right |\\\nonumber
\leq& \sum_{j=2}^{\infty}\sum_{m=0}^{\infty} 2^{j-1}\frac{\left | q \right |^{m+1} }{\left | z \right | }\left ( \frac{1}{\left | q \right |^{m+1} } \right )^{j} \left ( \frac{1}{\left | q \right | }  \right )^{j}\frac{1}{\left | z \right |^m }\\\nonumber     
\leq& \sum_{j=2}^{\infty}\sum_{m=0}^{\infty}2^{j-1}\left ( \frac{1}{\left | q \right |^{m+1} }  \right )^j\frac{1}{\left | q \right |^j }\\\nonumber
=&\sum_{j=2}^{\infty}\frac{1}{2}\left ( \frac{2}{\left | q \right | }  \right )^j\sum_{m=0}^{\infty}\left ( \frac{1}{\left | q \right |^{m+1} }  \right )^j\\\nonumber
\leq &\sum_{j=2}^{\infty}\frac{1}{2}\left ( \frac{2}{\left | q \right | }  \right )^j\frac{2}{\left | q \right |^j }\\\nonumber
\leq& \frac{1}{\left | q \right | }               
\end{align}
for all $z\in D(\rho).$
Also, from \eqref{E9} we can see that $$\left | P_{mj}(z) \right |=\left | \lambda^mc_j\left ( \frac{z}{q^{m+1}}  \right )y\left ( \frac{z}{q^{m+1}}  \right )^j    \right |\leq 2^{j-1}\left ( \frac{1}{\left | q \right |^{m+2} }  \right )^j, $$
and
$$\sum_{j=2}^{\infty}\sum_{m=0}^{\infty}2^{j-1}\left ( \frac{1}{\left | q \right |^{m+2} }  \right )^j=\sum_{j=2}^{\infty}\frac{1}{2}\left ( \frac{2}{\left | q \right | }  \right )^j\sum_{m=0}^{\infty}\left ( \frac{1}{\left | q \right |^{m+1} }  \right )^j<\infty.$$
By the Weierstrass M-test, $T[y](z)$ 
 converges absolutely and uniformly in $D(\rho),$ and so $T[y](z)$ is analytic in $D(\rho).$ By combining \eqref{E9}, we complete the proof that $T:X\to X.$

$\mathbf{Step 3:}$ We now proceed to prove that $T$ is a contraction mapping. Let $y,h \in X.$ Then by combining \eqref{E18}, \eqref{E7}, \eqref{E12}, \eqref{E8}, \eqref{E22}, \eqref{E14}, and \eqref{E31},
it follows that
\begin{align}
&\left | T[y](z)-T[h](z) \right |\\\nonumber
=& \left | \sum_{j=2}^{\infty}\sum_{m=0}^{\infty}\lambda^m c_j \left ( \frac{z}{q^{m+1}}  \right )\left ( y\left ( \frac{z}{q^{m+1}}  \right )^j-h\left ( \frac{z}{q^{m+1}}  \right )^j  \right )   \right |\\\nonumber
\leq &\sum_{j=2}^{\infty}\sum_{m=0}^{\infty}2^{j-1}\left | \frac{1}{q^{m+1}}  \right |\left | \left | y-h \right |  \right |\left ( \frac{1}{\left | q \right |^{m+2} }  \right )^{j-1}j\\\nonumber     
=&\sum_{j=2}^{\infty}\left ( \frac{2}{\left | q \right | }  \right )^{j-1}j\sum_{m=0}^{\infty}\frac{1}{\left | q \right |^{(m+1)j} }\left | \left | y-h \right |  \right |\\\nonumber
\leq &\sum_{j=2}^{\infty}\left ( \frac{2}{\left | q \right | }  \right )^{j-1}j\frac{2}{\left | q \right |^j } \left | \left | y-h \right |  \right |\\\nonumber
=&\sum_{j=2}^{\infty}\left ( \frac{2}{\left | q \right |^2 }  \right )^{j-1}j\frac{2}{\left | q \right | }\left | \left | y-h \right |  \right |\\\nonumber
=&\frac{2}{\left | q \right | }\frac{4\left | q \right |^2-4 }{\left ( \left | q \right |^2-2  \right )^2 }\left | \left | y-h \right |  \right |  \\\nonumber 
< &\frac{2}{\left | q \right | }\left | \left | y-h \right |  \right |.       \end{align}
Therefore, \begin{align}\label{E10}
    \left | \left | T[y](z)-T[h](z) \right |  \right |\leq\frac{2}{3}  \left | \left | y-h \right |  \right |. 
\end{align}
That is, $T$ is a contraction mapping.

$\mathbf{Step 4:}$
We now prove that \eqref{E1} admits a global meromorphic solution in $\mathbb{C}.$ 
By Lemma \ref{L2}, we conclude that $T$ is a contraction mapping on $X$ and so has a unique fixed point. That is, there exists an analytic function $y\in X$ such that \begin{align*}
y(z)=T[y](z)=\sum_{j=2}^{\infty}\sum_{m=0}^{\infty} \lambda^mc_j\left ( \frac{z}{q^{m+1}}  \right )y\left ( \frac{z}{q^{m+1}}  \right )^j
\end{align*}
is a solution of equation \eqref{E1} in $D(\rho).$ On the other hand, by the assumption in \eqref{E18}, the local meromorphic solution of \eqref{E1} in $D(\rho)$ can be extended meromorphically to the entire complex plane $\mathbb{C}.$
\end{proof}

\section{proof of theorem \ref{T2}}

\begin{proof}
Some details are the same as those in Step 1 of the proof of Theorem \ref{T1}, to which we refer.
We do not prove these details here.
Let $\rho =\frac{\left | q \right | }{2}.$ Then 
\begin{align}\label{E19}
\left | z \right |\geq \left | \Re(z) \right | \geq \rho=\frac{\left | q \right | }{2}>1.   
\end{align}
This means that all the coefficient functions  $a_j(z),$ $b_k(z)$ have no poles in $D(\rho).$
Since $\lambda =a_1(z)\equiv 0,$  by referring to \eqref{E3} we have \begin{align*}
y(qz)=\sum_{j=2}^{\infty}c_j(z)y(z)^j. 
\end{align*}
Then \begin{align}\label{E15}
y(z)=\sum_{j=2}^{\infty}c_j\left ( \frac{z}{q}  \right )y\left ( \frac{z}{q}  \right )^j.   
\end{align}
Similarly, referring to \eqref{E7},
we have \begin{align}\label{E16}
\left | c_j(z) \right |&\leq \sum_{1\leq k\leq j-1}\underset{i  = 1,\dots,p}{\max}\left | a_i(z) \right |\underset{i  = 1,\dots,t}\max\left | b_i(z) \right |^{k-1}+\left | a_1(z) \right |\left | b_1(z) \right |^{j-1}\\\nonumber
&\leq  \left ( 2^{j-1}-1 \right )\frac{1}{\left | z \right | }\\\nonumber
&<2^{j-1}\frac{1}{\left | z \right | }.    
\end{align}
Define an operator $T$ by 
\begin{align}\label{E17}
T[y](z)=\sum_{j=2}^{\infty}c_j\left ( \frac{z}{q}  \right )y\left ( \frac{z}{q}  \right )^j=\sum_{j=2}^{\infty}P_j(z) .   
\end{align}
Let $X$ be the set of all functions $z\to y(z),$ analytic and bounded in $$D(\rho)=\left \{ z:\left | \Re(z) \right | \geq \rho  \right \},$$ and let \begin{align}\label{E21} 
\left | \left | y \right |  \right |=\underset{z\in D(\rho)}{\sup }\left | y(z) \right |\leq \frac{1}{\left | q \right | }.
\end{align}
The proof that $X$ is a Banach space under the super-norm is given in the proof of Theorem \ref{T1}; we do not repeat it here. 
We will next prove that $T$ is a contraction mapping. We first prove that $T: X\to X.$
Assume that \eqref{E13} holds. Together with \eqref{E13}, \eqref{E19}, \eqref{E16}, and \eqref{E21}, we obtain
\begin{align}\label{E20}
\left | P_j(z) \right |=&\left | c_j\left ( \frac{z}{q}  \right )y\left ( \frac{z}{q}  \right )^j   \right |\\\nonumber
  \leq & \frac{\left | q \right | }{\left | z \right | }2^{j-1}\frac{1}{\left | q \right |^j } \frac{1}{\left | q \right |^j }\\\nonumber
=&\frac{\left | q \right | }{2\left | z \right | }\left ( \frac{2}{\left | q \right |^2 }  \right )^j\\\nonumber
\leq& \left ( \frac{2}{\left | q \right |^2 }  \right )^j      
\end{align}
for all $z\in D(\rho).$ 
Since $\sum_{j=2}^{\infty}\left ( \frac{2}{\left | q \right |^2 }  \right )^j  $ converges, it follows from \eqref{E17} and \eqref{E20}, by the Weierstrass M-test, that $T[y](z)$ converges absolutely and uniformly in $D(\rho),$ so $T[y](z)$ is analytic in $D(\rho).$
On the other hand,  \eqref{E17} and \eqref{E20} together yield
\begin{align}
\left | T[y](z) \right |=\left | \sum_{j=2}^{\infty}P_j(z) \right |  \leq \sum_{j  = 2}^{\infty}\left ( \frac{2}{\left | q \right |^2 }  \right )^j\leq \frac{4}{\left | q \right |^{4}-2\left | q \right |^{2}  }\leq \frac{1}{\left | q \right | }.    \end{align}
Therefore, $T$ maps $X$ to itself. We now proceed to prove that $T$ is a contraction mapping. Let $y,h\in X.$ Then by combining \eqref{E13}, \eqref{E19}, \eqref{E16}, \eqref{E17}, and \eqref{E21}, we have 
\begin{align*}
&\left | T[y](z)-T[h](z) \right |\\\nonumber
\leq& \sum_{j=2}^{\infty}\left | c_j \left ( \frac{z}{q}  \right ) \right |\left | y\left ( \frac{z}{q}  \right )-h\left ( \frac{z}{q}  \right )   \right |\left | y\left ( \frac{z}{q}  \right )^{j-1}+\dots+h\left ( \frac{z}{q}  \right )^{j-1}     \right |\\
\leq &\sum_{j=2}^{\infty}2^{j-1}\frac{\left | q \right | }{\left | z \right | }\left | \left | y-h \right |  \right |\frac{1}{\left | q \right | }\frac{1}{\left | q \right |^{j-1} }j\frac{1}{\left | q \right |^{j-1} }\\
=&\frac{1}{\left | z \right | }\sum_{j=2}^{\infty}j\left ( \frac{2}{\left | q \right |^2 }  \right )^{j-1}\left | \left | y-h \right |  \right |\\               
< &\frac{1}{\left | z \right | }\left | \left | y-h \right |  \right |.    
\end{align*}
This implies that $$\left | \left | T[y](z)-T[h](z) \right |  \right |\leq \frac{1}{\left | z \right | }\left | \left | y-h \right |  \right | .$$
From \eqref{E19}, we find that $\frac{1}{\left | z \right | }<1.$
Hence, $T$ is a contraction mapping. By referring to Step $4$ of the proof of Theorem \ref{E1}, we conclude that \begin{align*}
y(z)=T[y](z)=\sum_{j=2}^{\infty}c_j\left ( \frac{z}{q}  \right )y\left ( \frac{z}{q}  \right )^j 
\end{align*}
is a meromorphic solution to \eqref{E11} in the complex plane $\mathbb{C}.$
\end{proof}

\section{proof of theorem \ref{T3}}
\begin{lemma}\cite{Gundersen2002}\label{L4}
If $\left | q \right |>1 ,$ then any local meromorphic solution to \eqref{E40}
around the origin has a meromorphic continuation over the whole complex plane.    
\end{lemma}
\begin{proof}
Let $\rho^2+\sigma ^2\leq \left | q \right |^2  .$ Then \begin{align}\label{E27} 
\left | z \right |=\sqrt{\left | \Re(z) \right |^2 +\left | \Im(z) \right |^2 }\leq \sqrt{\rho^2+\sigma^2}  \leq \left | q \right |  .
\end{align} This means that all the coefficient functions $a_j(z),$ $b_k(z)$ have no poles in $D(\rho, \sigma).$
Referring to Step $1$ of the proof of Theorem \ref{T1}, we find that  \begin{align}
y(z)=\sum_{j=2}^{\infty}\sum_{m=0}^{\infty}\lambda ^m c_j\left ( \frac{z}{q^{m+1}}  \right ) y\left ( \frac{z}{q^{m+1}}  \right )^j, 
\end{align}
and \begin{align}\label{E24}
\left | c_j(z) \right |&\leq \sum_{1\leq k\leq j-1}\underset{i  = 1,\dots,p}{\max}\left | a_i(z) \right |\underset{i  = 1,\dots,t}\max\left | b_i(z) \right |^{k-1}+\left | a_1(z) \right |\left | b_1(z) \right |^{j-1}\\\nonumber
&<  \left ( 2^{j-1}-1 \right )\left | q \right |^{\left | z \right |(j-1)}+\left | q \right |^{\left | z \right |(j-1)}  \\\nonumber
&=\left | q \right |^{\left | z \right |(j-1) }2^{j-1}.     
\end{align}
Define an operator $T$ by 
\begin{align}\label{E25}
T[y](z)=\sum_{j=2}^{\infty}\sum_{m=0}^{\infty}\lambda ^m c_j\left ( \frac{z}{q^{m+1}}  \right )y\left ( \frac{z}{q^{m+1}}  \right )^{j}=\sum_{j=2}^{\infty}\sum_{m=0}^{\infty}P_{mj}(z)    
\end{align}
Let $X$ be the set of all functions $z\to y(z),$ analytic and bounded in $$D(\rho,\sigma )=\left \{ z:\left | \Re(z) \right |\leq \rho,\left | \Re(z) \right |\leq \sigma    \right \},$$ and let \begin{align}\label{E26} 
\left | \left | y \right |  \right |=\underset{z\in D(\rho,\sigma)}{\sup }\left | y(z) \right |\leq \frac{1}{\left | q \right | }.
\end{align}
The proof that $X$ is a Banach space under the super-norm is given in the proof of Theorem \ref{T1}; we omit it here. We now prove that $T: X\to X.$ Suppose that \eqref{E23} holds.   Together, \eqref{E23}, \eqref{E27}, \eqref{E24}, \eqref{E25}, \eqref{E26}, and the fact that 
\begin{align}\label{E28}
\frac{\left | q \right |^{j+1} }{\left | q \right |^{j+1}-1 }<\left ( \frac{\left | q \right | }{4}  \right )^j,\,\,\,\left | q \right |\geq 6,   
\end{align}
imply that
\begin{align}\label{E29}
&\left | T[y](z) \right |\\\nonumber
=& \left | \sum_{j=2}^{\infty}\sum_{m=0}^{\infty}P_{mj}(z) \right |       \\\nonumber
\leq &   \sum_{j=2}^{\infty}\sum_{m=0}^{\infty}\left | \lambda ^m c_j\left ( \frac{z}{q^{m+1}}  \right )y\left ( \frac{z}{q^{m+1}}  \right )^j  \right |\\\nonumber
\leq &\sum_{j=2}^{\infty}\sum_{m=0}^{\infty}\frac{1}{\left | q \right |^{m} }2^{j-1}\left | q \right |^{j-1}\left | \frac{1}{q^{m+1}}  \right |^{j}\frac{1}{\left | q \right |^j }\\\nonumber
=&\sum_{j=2}^{\infty}\frac{1}{2\left | q \right | }\left ( \frac{2 }{\left | q \right | }  \right )^j\sum_{m=0}^{\infty}\left ( \frac{1}{\left | q \right |^{j+1} }  \right )^m\\\nonumber
<& \sum_{j=2}^{\infty}\frac{1}{2\left | q \right | }\left ( \frac{2}{\left | q \right | }  \right )^{j}\left ( \frac{\left | q \right | }{4}  \right )^{j}      \\\nonumber
=&\frac{1}{4\left | q \right | }<\frac{1}{\left | q \right | }. \end{align}
Also, from \eqref{E29} we can see that  
\begin{align*}
\left | P_{mj}(z) \right |=\left | \lambda ^{m} c_{j}\left ( \frac{z}{q^{m+1}}  \right )y\left ( \frac{z}{q^{m+1}}  \right )^j\right |    \leq \left ( \frac{1}{\left | q \right |^{j+1} }  \right )^{m}\frac{1}{2\left | q \right | }\left ( \frac{2}{\left | q \right | }  \right )^j    
\end{align*}
Since \begin{align}
\sum_{m=0}^{\infty}\left ( \frac{1}{\left | q \right |^{j+1} }  \right )^m \frac{1}{2\left | q \right | }\sum_{j=2}^{\infty}    
\left ( \frac{2}{\left | q \right | }  \right )^j<\frac{1}{4\left | q \right | }  <\infty,
\end{align}
By the Weierstrass M-test, we conclude that $T[y](z)=\sum_{j=2}^{\infty}\sum_{m=0}^{\infty}P_{mj}(z)$ converges absolutely and uniformly in $D(\rho, \sigma),$ and hence $T[y](z)$ is analytic in $D(\rho, \sigma).$ Therefore, $T: X\to X.$ Next, we prove that $X$ is a contraction mapping. Let $y,h\in X.$ Combining \eqref{E23}, \eqref{E27}, \eqref{E24}, \eqref{E25}, \eqref{E26}, and \eqref{E28} yield
\begin{align*}
&\left | T[y](z)-T[h](z) \right |\\\nonumber
\leq & \sum_{j=2}^{\infty}\sum_{m=0}^{\infty}\left | \lambda^m c_j\left ( \frac{z}{q^{m+1}}  \right )\left ( y\left ( \frac{z}{q^{m+1}}  \right )^j-h\left ( \frac{z}{q^{m+1}}  \right )^j  \right )  \right |\\
\leq& \sum_{j=2}^{\infty}\sum_{m=0}^{\infty}\frac{1}{\left | q \right |^m }2^{j-1}\left | q \right |^{j-1}\frac{1}{\left | q \right |^{m+1} }\left ( \frac{1}{\left | q \right |^{m+2} }  \right )^{j-1}j\left | \left | y-h \right |  \right |\\ 
<&\sum_{j=2}^{\infty}\frac{1}{2}\left ( \frac{2}{\left | q \right | }  \right )^jj\left ( \frac{\left | q \right | }{4}  \right )^j\left | \left | y-h \right |  \right |\\     
=&\sum_{j=2}^{\infty}\frac{1}{2}\left ( \frac{1}{2}  \right )^j j\left | \left | y-h \right |  \right |\\   
=&\frac{3}{4}\left | \left | y-h \right |  \right |.  
\end{align*} 
By Lemma \ref{L2}, we conclude that $X$ has a unique fixed point under contraction mapping $T.$ That is, there exists an analytic function $y\in X$ such that $$y(z)=T[y](z)=\sum_{j=2}^{\infty}\sum_{m=0}^{\infty}\lambda ^m c_j\left ( \frac{z}{q^{m+1}}  \right )y\left ( \frac{z}{q^{m+1}}  \right )^{j} $$
is a solution to equation \eqref{E40} in $D(\rho, \sigma).$ Then we can extend the local meromorphic solution $y(z)$ to the whole complex plane $\mathbb{C}$ by Lemma~\ref{L4}.
\end{proof}

\section{proof of theorem~\ref{T4}}
\begin{proof}
Some details are the same as those in the proof of Theorem \ref{T2}, to which we refer.
We omit these details here.
Let $\rho^2+\sigma ^2\leq \left | q \right |^2  .$ Then \begin{align}\label{E37} 
\left | z \right |=\sqrt{\left | \Re(z) \right |^2 +\left | \Im(z) \right |^2 }\leq \sqrt{\rho^2+\sigma^2}  \leq \left | q \right |  .
\end{align} This means that all the coefficient functions $a_j(z),$ $b_k(z)$ have no poles in $D(\rho, \sigma).$
 Referring to \eqref{E15} and \eqref{E16}, we have \begin{align}\label{E34}
y(z)=\sum_{j=2}^{\infty}c_j\left ( \frac{z}{q}  \right )y\left ( \frac{z}{q}  \right )^j,   
\end{align}
and \begin{align}\label{E35}
\left | c_j(z) \right |&\leq \sum_{1\leq k\leq j-1}\underset{i  = 1,\dots,p}{\max}\left | a_i(z) \right |\underset{i  = 1,\dots,t}\max\left | b_i(z) \right |^{k-1}+\left | a_1(z) \right |\left | b_1(z) \right |^{j-1}\\\nonumber
&<2^{j-1}\left | q \right |^{\left | z \right |(j-1) }.     
\end{align}
Define \begin{align}\label{E36}
T\left [ y \right ](z)  = \sum_{j  = 2}^{\infty}c_{j}\left ( \frac{z}{q}  \right )y\left ( \frac{z}{q}  \right )^{j}  = \sum_{j  = 2}^{\infty}E_j(z). \end{align}
Let $X$ be the set of all functions $z\to y(z),$ analytic and bounded in $$D(\rho,\sigma )=\left \{ z:\left | \Re(z) \right |\leq \rho,\left | \Re(z) \right |\leq \sigma    \right \},$$ and let \begin{align}\label{E38} 
\left | \left | y \right |  \right |=\underset{z\in D(\rho,\sigma)}{\sup }\left | y(z) \right |\leq \frac{1}{\left | q \right | }.
\end{align}
By referring to Step $2$ of the proof of Theorem~ \ref{T1}, we see that $X$ is a Banach Space. Suppose that \eqref{E33} holds.
By combining \eqref{E33}, \eqref{E37}, \eqref{E35}, \eqref{E36}, and \eqref{E38}, we have
$$\left | E_j(z) \right |=\left | c_j\left ( \frac{z}{q}  \right )y\left ( \frac{z}{q}  \right )^j   \right |\leq \frac{1}{2\left | q \right | }\left ( \frac{2}{\left | q \right | }  \right )^{j}    .$$
Note that $$\sum_{j=2}^{\infty}\frac{1}{2\left | q \right | }\left ( \frac{2}{\left | q \right | }  \right )^j\leq \frac{1}{\left | q \right | } <\infty.$$ 
Then by the Weierstrass-M test, we have $T[y](z)$ is analytic in $D(\rho, \sigma),$ and $$\left | T\left [ y \right ](z)  \right |\leq \sum_{j=2}^{\infty}\frac{1}{2\left | q \right | }\left ( \frac{2}{\left | q \right | }  \right )^{j}\leq \frac{1}{\left | q \right | }.$$ This shows that $T$ maps $X$ to itself. Moreover,
\begin{align*}
&\left | T[y](z)-T[h](z) \right |\\
=&\left | \sum_{j=2}^{\infty}c_j\left ( \frac{z}{q}  \right )\left ( y\left ( \frac{z}{q}  \right )^{j}- h\left ( \frac{z}{q}  \right )^{j} \right )    \right |\\  
\leq &\sum_{j=2}^{\infty}\left ( \frac{2}{\left | q \right | }  \right )^{j-1}j\frac{1}{\left | q \right | }\left | \left | y-h \right |  \right | \\  
<&\frac{1}{4}\left | \left | y-h \right |  \right |.  
\end{align*}
Then $$\left | \left | T[y](z)-T[h](z) \right |  \right |\leq \frac{1}{4}\left | \left | y-h \right |  \right |   .$$
By Lemma \ref{L2}, we conclude that $X$ has a unique fixed point under contraction mapping $T.$ That is, there exists an analytic function $y\in X$ such that $$y(z)=T[y](z)=\sum_{j=2}^{\infty}c_j\left ( \frac{z}{q}  \right )y\left ( \frac{z}{q}  \right )^{j}$$   
is a solution to equation \eqref{E39} in $D(\rho, \sigma).$ Then we can extend the local meromorphic solution $y(z)$ to the whole complex plane $\mathbb{C}$ by Lemma~\ref{L4}.

\end{proof}

\section{acknowledgements}
Before submitting the manuscript, I had some valuable discussions with Professor Risto Korhonen about applying the Banach fixed point theorem to examine the existence of analytic solutions to complex difference equations. I would like to thank Professor Korhonen for his valuable guidance.

%\section{acknowledgements}



%\bibliographystyle{siam}
%\bibliography{name}
\end{document}